\documentclass[twoside,11pt]{amsart}
\usepackage{amsmath,amssymb,hyperref}

\usepackage{verbatim}

\def\sqr#1#2{{\vcenter{\hrule height.#2pt
        \hbox{\vrule width.#2pt height#1pt \kern#1pt
                \vrule width.#2pt}
        \hrule height.#2pt}}}

\usepackage{pst-plot}
\usepackage{pst-node,pst-text,pst-3d,pstricks}
\def\sqr#1#2{{\vcenter{\hrule height.#2pt
        \hbox{\vrule width.#2pt height#1pt \kern#1pt
                \vrule width.#2pt}
        \hrule height.#2pt}}}
\def\square{\mathchoice\sqr64\sqr64\sqr{4}3\sqr{3}3}
\def\QED{\hfill$\square$}

\newtheorem{Theorem}{Theorem}[section]

\newtheorem{Lemma}[Theorem]{Lemma}

\newtheorem{Proposition}[Theorem]{Proposition}

\newtheorem{Assumptions and Discussion}[Theorem]{Assumptions and Discussion}

\newtheorem{Remark}[Theorem]{Remark}

\newtheorem{Definition}[Theorem]{Definition}

\theoremstyle{definition}
\newtheorem*{acknowledgement}{Acknowledgement}

\numberwithin{equation}{section}

\newcommand{\ara}{\mathop{\mathrm{ara}}\nolimits}
\newcommand{\biara}{\mathop{\mathrm{biara}}\nolimits}
\newcommand{\triara}{\mathop{\mathrm{triara}}\nolimits}
\newcommand{\pd}{\mathop{\mathrm{pd}}\nolimits}
\newcommand{\height}{\mathop{\mathrm{height}}\nolimits}
\renewcommand{\H}{\mathcal{H}}

\newcommand{\s}{\mathcal{S}}

\begin{document}

\baselineskip=16pt

\title[Arithmetical Rank of Strings and Cycles]
{\Large\bf Arithmetical Rank of Strings and Cycles}

\thanks{AMS 2010 {\em Mathematics Subject Classification}:
Primary 13F55; Secondary 13A15.}
%Primary 13A30; Secondary 13H15, 13B22, 13C14, 13C15, 13C40.}
\thanks{$^1$ K. Kimura was partially supported by JSPS Grant-in-Aid for 
  Young Scientists (B) 24740008. }
\thanks{$^2$ P. Mantero was partially supported by an AMS-Simons Travel Grant.}

\author[Kyouko Kimura]{Kyouko Kimura$^1$}
\address{Department of Mathematics, Graduate School of Science, Shizuoka University,
836 Ohya, Suruga-ku, Shizuoka 422-8529, Japan}
\email{skkimur@ipc.shizuoka.ac.jp}

\author[Paolo Mantero]{Paolo Mantero$^2$}
\address{University of California, Riverside \\ Department of Mathematics \\
 Riverside, CA 92521}
\email{mantero@math.ucr.edu\newline
\indent{\it URL:} \href{http://math.ucr.edu/~mantero/}{\tt http://math.ucr.edu/$\sim$mantero/}}

\date{\today}
%\date{October 20, 2013}

\vspace{-0.1in}

%\begin{abstract}
%
%\end{abstract}

\maketitle

\vspace{-0.2in}

\begin{abstract}
  Let $R$ be a polynomial ring over a field $K$. 
  To a given squarefree monomial ideal $I \subset R$, 
  one can associate a hypergraph $\H(I)$. 
  In this article, we prove that the arithmetical rank of $I$ 
  is equal to the projective dimension of $R/I$ when 
  $\H (I)$ is a string or a cycle hypergraph. 
\end{abstract}

%%%%%%%%%%%%%%%%%%%%%%%%%%%%%%%%%%%
\section*{Introduction}
Let $R=K[x_1, \ldots, x_n]$ be a polynomial ring over a field $K$ 
and $I$ a squarefree monomial ideal of $R$. 
The \textit{arithmetical rank} of $I$, denoted by $\ara I$, is defined 
as the minimum number $u$ of elements $q_1, \ldots, q_u \in R$ such that the equality
\begin{displaymath}
  \sqrt{(q_1, \ldots, q_u)} = \sqrt{I} \;(= I)
\end{displaymath}
holds. When this is the case, one says that $q_1, \ldots, q_u$ 
generate $I$ up to radical. 
Let $G (I)$ denote the minimal set of monomial generators of $I$ and set $\mu (I) = \# G(I)$. 
Then $\ara I \leq \mu (I)$ holds. On the other hand, Lyubeznik \cite{Ly} 
proved that $\ara I \geq \pd R/I$, where $\pd R/I$ denotes 
the projective dimension of $R/I$. Therefore we have 
\begin{displaymath}
  \height I \leq \pd R/I \leq \ara I \leq \mu (I). 
\end{displaymath}
From the above inequalities, 
it is natural to ask when $\ara I = \pd R/I$ holds. 
Many authors including \cite{Barile, BKMY, BTcone, BTpolygon, EOT, Kimura, KRT, 
KTforest, KTh3Gor, KTY, KTY2, Kummini, Macchia, Morales, SV, Varbaro} 
investigated this problem. 
In particular, in \cite{KTY, KTY2} (see also \cite{KRT}), 
Terai, Yoshida and the first author 
attacked the problem for ideals $I$ with $\mu (I) - \height I \leq 2$. 
Their idea is to classify these squarefree monomial ideals using hypergraphs 
(this classification is also used in \cite{KTY3}). The association of a hypergraph to a squarefree monomial ideal $I$ of $R$
with $G(I) = \{ m_1, \ldots, m_{\mu} \}$ is defined by setting 
\begin{displaymath}
  \H (I) := \big\{ \{ j \in [\mu] \; : \; x_i \mid m_j \} 
               \; : \; i=1, \ldots, n \big\}.
\end{displaymath}
$\H(I)$ is indeed a (separated) hypergraph on the vertex set 
$[\mu] := \{ 1, 2, \ldots, \mu \}$. 
On the other hand, given a separated hypergraph $\H$, 
one can construct a squarefree monomial ideal $I$ with $\H (I) = \H$; 
see Section \ref{sec:hypergraph} for more details. 

\par
We focus on the squarefree monomial ideals $I$ 
such that $\H (I)$ is a string or a cycle. 
For these ideals, Lin and  the second author \cite{LM} found an explicit formula 
expressing the projective dimension of $R/I$ in terms of purely combinatorial invariants of the hypergraph $\H (I)$, namely
$$\pd(R/I)= \mu(I) - b(\H(I))+M(\H(I)).$$
See the discussion before Theorem \ref{LMpd} for the definition of $b(\H(I))$ and $M(\H(I))$.

In the present work we study the arithmetical rank of these ideals $I$. We prove that $\pd R/I$ elements can be chosen so that they generate $I$ up to radical, and have ``small'' monomial support. To be more precise, let us recall that the \textit{binomial arithmetical rank} of $I$, denoted by $\biara I$, is the minimum number $u$ of binomials or monomials 
$q_1, \ldots, q_u \in R$ which generate $I$ up to radical. 
Here we also define the \textit{trinomial arithmetical rank} of $I$ 
as the minimum number $u$ of trinomials, binomials or monomials 
$q_1, \ldots, q_u \in R$ which generate $I$ up to radical. 
We denote it by $\triara I$. 
Clearly one has $\ara I \leq \triara I \leq \biara I$. 
Our main result is the following theorem. 
\begin{Theorem}
  \label{MainThm}
  Let $I$ be a squarefree monomial ideal of $R$. 
  \begin{enumerate}
  \item Assume that $\H (I)$ is a string hypergraph. 
    Then $\ara I = \biara I = \pd R/I$. 
  \item Assume that $\H (I)$ is a cycle hypergraph. 
    Then $\ara I = \triara I = \pd R/I$. 
  \end{enumerate}
\end{Theorem}
In particular, the arithmetical rank of these ideals is independent of the 
characteristic of the field $K$. 
Crucial ingredients of our proof of Theorem \ref{MainThm} 
are a lemma by Schmitt and Vogel (\cite{SV}, Lemma \ref{SVLemma}) 
and the above formula for the projective dimension (Theorem \ref{LMpd}).

\par
Now we explain the organization of this article. 
In Section \ref{sec:hypergraph}, we recall the definition of the (separated) hypergraph 
associated to a squarefree monomial ideal, first introduced in \cite{KTY}. 
In Section \ref{sec:pd}, we recall a few results by Lin and the second author \cite{LM} that will be employed in the subsequent sections.
%about the projective dimension of a string hypergraph and a cycle hypergraph, 
%which we use in the proof of Theorem \ref{MainThm}. 
Then, in Sections \ref{sec:string} and  \ref{sec:cycle}, 
we prove Theorem \ref{MainThm} (1) and (2), respectively. 

%%%%%%%%%%%%%%%%%%%%%%%%%%%%%%%%%%%
\section{Hypergraphs}
\label{sec:hypergraph}
In this section, we recall the construction of a separated hypergraph associated to any 
squarefree monomial ideal. The construction was introduced in 
\cite{KTY}; see also \cite{KRT, KTY2, KTY3, LM}. 
\smallskip

\par
Set $V=[\mu]$. 
A collection $\H \subset 2^V$ is called a \textit{hypergraph} 
on the vertex set $V$ if $V = \bigcup_{F \in \H} F$. 
An element $F \in \H$ is called a face of $\H$. 
A vertex $j \in V$ is called closed (resp.\  open) if $\{ j \} \in \H$ 
(resp.\  $\{ j \} \notin \H$). 
A hypergraph is called saturated if $\{ j \} \in \H$ for all $j \in V$. 
Let $i,j \in V$ be two vertices of $\H$. We say that $i$ is a neighbor of $j$ 
if there exists a face $F \in \H$ containing both $i$ and $j$. 

\par
A hypergraph $\H$ on $V$ is said to be separated if 
for all vertices $i,j \in V$ ($i \neq j$), there exist faces $F, G \in \H$ 
such that $i \in F \setminus G$ and $j \in G \setminus F$. 
Let $I$ be a squarefree monomial ideal of $R = K[x_1, \ldots, x_n]$ 
with $G(I) = \{ m_1, \ldots, m_{\mu} \}$. 
The hypergraph associated to $I$ is defined as  
\begin{displaymath}
  \H (I) := \big\{ \{ j \in [\mu] \; : \; x_i \mid m_j \} 
               \; : \; i=1, \ldots, n \big\}, 
\end{displaymath}
which is a separated hypergraph on $[\mu]$. 

\par
Conversely, let $\H$ be a separated hypergraph on $[\mu]$. 
Then we can construct a squarefree monomial ideal $I$ with $\H (I) = \H$ 
in a polynomial ring with enough variables as follows: 
for each $F \in \H$, take a squarefree monomial $m_F$ such that 
$m_F$ and $m_G$ are coprime if $F \neq G$. 
For each $j \in [\mu]$, set $m_j = \prod_{F \in \H, \, j \in F} m_F$. 
Then $I=(m_1, \ldots, m_{\mu})$ is a squarefree monomial ideal with 
$\H (I) = \H$. 
This construction implies that there are many ideals $I$ 
(in various polynomial rings) with $\H (I) = \H$. 
We set $I (\H)$ to be the ideal obtained from the above construction 
by setting each $m_F$ to be a variable $x_F$ in a polynomial ring 
$R(\H) := K[x_F : F \in \H]$. 

\par
The above correspondence between squarefree monomial ideals 
and separated hypergraphs yields the classification 
of squarefree monomial ideals mentioned in the introduction. The following proposition shows 
the usefulness of this association for our purpose. 
\begin{Proposition}[{\cite[Corollary 2.4]{LM}, \cite[Proposition 3.2]{KRT}}]
  \label{class-hypergraph}
  Let $I_1, I_2$ be squarefree monomial ideals with $\H (I_1) = \H (I_2)$. 
  Then $\pd I_1 = \pd I_2$ and $\ara I_1 = \ara I_2$ hold. 
%  \begin{enumerate}
%  \item $\pd I_1 = \pd I_2$. 
%  \item $\ara I_1 = \ara I_2$. 
%  \end{enumerate}
\end{Proposition}
Let $I$ be a squarefree monomial ideal of $R$. Set $\H = \H (I)$. 
By Proposition \ref{class-hypergraph}, 
the following notation is well-defined: 
$\pd (\H) := \pd R/I$,  $\ara (\H) := \ara (I)$. 
We call $\pd (\H)$ (resp.\  $\ara (\H)$) the projective dimension 
(resp.\  arithmetical rank) of $\H$. 
We will compute $\pd (\H)$, $\ara (\H)$ by computing
$\pd R(\H)/I(\H)$, $\ara I (\H)$, respectively.  
%with re-labeling variables. 

\begin{Remark}
  \label{bi-tri-ara}
The statement of Proposition \ref{class-hypergraph}  remains true if we replace the arithmetical rank by
  the binomial or the trinomial arithmetical rank.
  Hence, we use the similar notations $\biara (\H)$, $\triara (\H)$.
\end{Remark}
%We have similar assertions for the binomial arithmetical rank 
%and the trinomial arithmetical rank to the arithmetical rank 
%and we use the similar notations. 

%%%%%%%%%%%%%%%%%%%%%%%%%%%%%%%%%%%
\section{Projective dimensions of a string hypergraph and a cycle hypergraph}
\label{sec:pd}
In this section, we collect results about the projective dimensions 
of a string hypergraph and a cycle hypergraph. 
These results are proved 
by Lin and the second author in \cite{LM}.

\par
We first recall the definitions of a string hypergraph and a cycle hypergraph. 
\begin{Definition}[{\cite[Definition 2.13]{LM}}]
  \label{defn:string}
  Fix $\mu \geq 2$. 
  A hypergraph $\H$ on $V = [\mu]$ is a string 
  if $\{ j, j+1 \} \in \H$ for all $j = 1, \ldots, \mu - 1$ 
  and the only other possible faces of $\H$ are of the form $\{ j \}$, for some $j \in V$. 
\end{Definition}
For a string hypergraph $\H$ on $[\mu]$, we call the vertices $1$ and $\mu$ 
the endpoints of $\H$. Note that if $\H$ is separated, 
then both endpoints are closed vertices. 
\begin{Definition}[{\cite[Definition 4.1]{LM}}]
  \label{defn:cycle}
  Fix $\mu \geq 3$. 
  A hypergraph $\H$ on $V = [\mu]$ is a $\mu$-cycle if 
  $\H$ can be written as $\H = \widetilde{\H} \cup \{ \{ \mu, 1 \} \}$ where 
  $\widetilde{\H}$ is a string hypergraph on $[\mu]$. 
\end{Definition}

To introduce the explicit formula for the projective dimension of 
a string hypergraph and a cycle hypergraph in terms of invariants of the hypergraph 
we need some more definitions. 

\par
A hypergraph $\H$ on $[\mu]$ is called a \textit{string of opens} 
if $\H$ is a string hypergraph with $\mu \geq 3$ whose only closed vertices 
are its endpoint. \smallskip

\par
First we assume that $\H$ is a string hypergraph. 
We set $s = s(\H)$ to be the number of strings of opens inside $\H$. 
We number the strings of opens in $\H$ from one endpoint to another and 
set $n_i (\H)$ to be the number of open vertices in the $i$-th string of opens. 
We say that $\H$ is a \textit{$2$-special configuration} if 
$s \geq 2$, $\H$ does not contain two adjacent closed vertices, 
$n_1 \equiv n_s \equiv 1 \mod 3$, 
and $n_i \equiv 2 \mod 3$ for $i=2, \ldots, s-1$. 
Two $2$-special configurations contained in $\H$ are said to be 
\textit{disjoint} if they do not have a common open vertex. 
The \textit{modularity} of $\H$, denoted by $M(\H)$, 
is the maximum number of pairwise disjoint $2$-special configurations 
contained in $\H$. 

\par
Next we assume that $\H$ is a cycle hypergraph.  If $\H$ contains at least two closed vertices, 
we define $s = s(\H)$ and $n_1 (\H), \ldots, n_s (\H)$ analogously to  
the case of a string hypergraph. 
If $\H$ contains at most one closed vertex, we set $s= s(\H) = 1$ and 
$n_1 (\H) = \mu (\H) - 1$.
In either case, the definition of a $2$-special configuration $\mathcal{S}$ in $\H$ 
is the same as in the case of a string hypergraph, except for allowing that the two extremal vertices of 
$\mathcal{S}$ coincide. 
The modularity $M(\H)$ is defined in the same way as in the case of 
a string hypergraph. 

\par
Let $\H$ be a string hypergraph or a cycle hypergraph. Set 
\begin{displaymath}
  b(\H) = s(\H) 
        + \sum_{i=1}^{s(\H)} \left\lfloor \frac{n_i (\H) - 1}{3} \right\rfloor. 
\end{displaymath}
\begin{Theorem}[{Lin and Mantero \cite[Theorems 3.4 and 4.3]{LM}}]
  \label{LMpd}
  Let $\H$ be a string hypergraph or a cycle hypergraph. 
  Then 
  \begin{displaymath}
    \pd (\H) = \mu (\H) - b(\H) + M(\H). 
  \end{displaymath}
\end{Theorem}

\par
We also collect some inductive results about the projective dimension. 
%Moreover for inductive argument in the our proof of Theorem \ref{MainThm}, 
%we need more about the projective dimension. 

\par
Let $I$ be a squarefree monomial ideal with 
$G(I) = \{ m_1, \ldots, m_{\mu} \}$. 
Then we set $I_i := (m_{i+1}, \ldots, m_{\mu})$ and $\H_i := \H (I_i)$. 
Also we set $J_1 := I_1 : m_1$ and $\mathcal{Q}_1 := \H (J_1)$. 
\begin{Lemma}[{\cite[Lemmas 2.6 and 2.11]{LM}}]
  \label{LMpd-inductive}
  Let $\H$ be a hypergraph on $[\mu]$ with $\mu \geq 2$. 
  Assume that $\{ 1 \} \in \H$. 
  Then $\pd (\H) = \max \{ \pd (\H_1), \pd (\mathcal{Q}_1) + 1 \}$. 
  Moreover, if all the neighbors of $1$ 
  are closed vertices, then $\pd (\H) = \pd (\H_1) + 1$. 
\end{Lemma}

Finally, for a string hypergraph $\H$, we will use the following results that allow us to compare $\pd(\H)$ with the projective dimension of a smaller string hypergraph.
\begin{Lemma}[{\cite[Lemma 2.14 (ii)]{LM}}]
  \label{LMpd-string-inductive-2}
  Let $\H$ be a string hypergraph on $[\mu]$ with $\mu \geq 3$.  
  Then $\pd (\H) \leq \pd (\H_2) + 2$. 
\end{Lemma}
\begin{Lemma}[{\cite[Proposition 2.15]{LM}}]
  \label{LMpd-string-inductive-3}
  Let $\H$ be a string hypergraph on $[\mu]$ with $\mu \geq 4$.  
  Assume $\{ 2 \} \notin \H$. Then $\pd (\H) = \pd (\H_3) + 2$. 
\end{Lemma}

%%%%%%%%%%%%%%%%%%%%%%%%%%%%%%%%%%%
\section{Strings}
\label{sec:string}
In this section, we consider string hypergraphs. 
The goal of this section is to prove the following result.

\begin{Theorem}\label{strings}
  Let $\H$ be a string hypergraph. Then $\ara (\H) = \biara (\H) = \pd (\H)$. 
\end{Theorem}

Before proving the theorem, we introduce a useful lemma 
by Schmitt and Vogel \cite{SV}. 
%Also, the following result by Schmitt and Vogel \cite{SV} is useful. 
\begin{Lemma}[{\cite[Lemma, p.\  249]{SV}}]
  \label{SVLemma}
  Let $R$ be a commutative ring and $P$ a finite subset of $R$. 
  Let $P_0, P_1, \ldots, P_u$ be subsets of $P$ satisfying the following 
  $3$ conditions: 
  \begin{enumerate}
  \item[(SV1)] $\bigcup_{\ell = 0}^u P_\ell = P$. 
  \item[(SV2)] $\# P_0 = 1$. 
  \item[(SV3)] For any integer $\ell >0$ and elements $p, p'' \in P_{\ell}$ 
    with $p \neq p''$, 
    there exist an integer $\ell' < \ell$ and an element $p' \in P_{\ell'}$ 
    such that $pp'' \in (p')$. 
  \end{enumerate}

  \par
  Let $I$ be an ideal of $R$ generated by $P$ and set 
  \begin{displaymath}
    q_{\ell} = \sum_{p \in P_{\ell}} p, \qquad \ell = 0, 1, \ldots, u. 
  \end{displaymath}
  Then $q_0, q_1, \ldots, q_{u}$ generate $I$ up to radical. 
\end{Lemma}

\par
We first see the case where the number of vertices is 
less than or equal to $3$. 
\begin{Lemma}\label{3}
%  If $\H$ is a string hypergraph with $\mu\leq 3$ vertices, 
%  then $\ara(\H)=\pd(\H)$, and $I(\H)=\sqrt{J}$ where 
%  $J$ is generated by monomial and binomials.
  Let $\H$ be a string hypergraph on $[\mu]$. 
  If $\mu \leq 3$, then $\ara (\H) = \biara (\H) = \pd (H)$. 
%  If $\H$ is a string hypergraph with $\mu\leq 3$ vertices, 
% then $\ara(\H)=\pd(\H)$, and $I(\H)=\sqrt{J}$ 
%  where $J$ is generated by monomial and binomials.
%  Moreover, $\ara (\H)$ elements which generate $I(\H)$ up to radical can be 
%  constructed in monomials and binomials. 
\end{Lemma}
\begin{proof}
  If $\H$ is saturated, then $\pd (\H) = \mu$ and there is nothing to prove. 
  The remaining case is that $\mu = 3$ and the vertex $2$ of $\H$ is open. 
  Then $I (\H) = (y_1 x_1, x_1 x_2, y_3x_2)$. 
  In this case $\pd (\H ) =2$. By Lemma \ref{SVLemma}, 
  we have $x_1x_2, y_1 x_1+ y_3 x_2$ generate $I(\H)$ up to radical. 
\end{proof}

Next we assume $\mu \geq 4$. We divide the proof into two cases, depending on whether the vertex $2$ 
is closed or open. 
\begin{Lemma}\label{+1}
  Let $\H$ be a string hypergraph on $[\mu]$. 
  Assume the neighbor $2$ of the endpoint $1$ of $\H$ is closed.
%  Let $\H$ be a string with $\mu\geq 2$ vertices, 
%  and assume the neighbor $2$ of the endpoint $1$ of $\H$ is closed.
  If $\biara (\H_1) = \pd (\H_1)$, then $\biara (\H) = \pd (\H)$.
%  If $\ara(\H_1)=\pd(\H_1)$, then $\ara(\H)=\pd(\H)$.
\end{Lemma}
\begin{proof}
  We first note that $\biara (\H) \leq \biara (\H_1) + 1$ 
  since $I(\H)$ has one more generator than $I(\H_1)$. 
%  then $\ara(\H)\leq \ara(\H_1)+1$.
  We then have the chain of inequalities
  $$\biara (\H) \leq \biara (\H_1) + 1 
    = \pd (\H_1) + 1 = \pd (\H) \leq \biara (\H),$$
%  $$\ara (\H) \leq \ara (\H_1) + 1 = \pd (\H_1) + 1 = \pd (\H) \leq \ara (\H).$$
%  where the last equality follows by \cite[Lemma~2.11]{LM}. 
%  Therefore, $\ara(\H)=\pd(\H)$.
  where the last equality follows by Lemma \ref{LMpd-inductive}. 
  Therefore, $\biara (\H) = \pd (\H)$.
\end{proof}

\begin{Lemma}\label{+2}
  Let $\H$ be a string hypergraph on $[\mu]$ with $\mu \geq 4$. 
  Assume the neighbor $2$ of the endpoint $1$ of $\H$ is open.
%  Let $\H$ be a string with $\mu\geq 3$ vertices, 
%  and assume the neighbor $2$ of the endpoint $1$ of $\H$ is open.
%  If $\ara(\H_3)=\pd(\H_3)$, then $\ara(\H)=\pd(\H)$.
  If $\biara (\H_3) = \pd (\H_3)$, then $\biara (\H) = \pd (\H)$.
\end{Lemma}
\begin{proof}
  Write $I(\H) = I_3 + I'$ where $I_3 = I(\H_3) = (m_4, \ldots, m_{\mu})$ 
  and $I'=(m_1,m_2,m_3)$. 
  Note that $\H (I')$ is a string hypergraph on the vertex set $[3]$. 
  Since the vertex $2$ of $\H (I')$ is open, we have $\biara I' = 2$ 
  by Lemma \ref{3}. 
%  By assumption, $m_1=x_1x_2$, $m_2=x_2x_3$ 
%  and either $m_3=x_3x_4x_5$ and $I(\H_3)\subseteq k[x_5,\ldots,x_n]$ 
%  or $m_3=x_3x_4$ and $I(\H_3)\subseteq k[x_4,\ldots,x_n]$.
  We then have 
  $$\biara (\H) = \biara (I_3 + I') \leq \biara (I_3) + \biara (I') 
    = \biara (I_3) + 2 = \pd (\H_3) + 2.$$
  Since $\pd (\H_3) +2=\pd (\H)$ by Lemma \ref{LMpd-string-inductive-3}, 
  and $\pd (\H) \leq \biara (\H)$ always holds, we have 
  $\biara (\H) = \pd (\H)$.
%  $$\ara (\H) = \ara (I_3 + I') \leq \ara (I_3) + \ara(I') 
%     = \ara (I_3) + 2 = \pd (\H_3) + 2 = \pd (\H) \leq \ara (\H),$$
%  where the last equality follows by \cite[Proposition~2.15]{LM}. 
%  Therefore, $\ara(\H)=\pd(\H)$.
\end{proof}

We can now prove Theorem \ref{strings}.\\
{\bf Proof of Theorem \ref{strings}.} 
  We prove it by induction on the number $\mu$ of vertices of $\H$.

  \par
  If $\mu \leq 3$, then the statement follows by Lemma \ref{3}. 
  We may then assume $\mu \geq 4$ and the statement is proved 
  for string hypergraphs with less than $\mu$ vertices. 
  Then both $\biara (\H_1) = \pd (\H_1)$ and $\biara (\H_3) = \pd (\H_3)$ hold, 
  and the assertion follows from Lemmas \ref{+1} and \ref{+2}. 

\QED
\bigskip

%%%%%%%%%%%%%%%%%%%%%%%%%%%%%%%%%%%
\section{Cycles}
\label{sec:cycle}
In this section, we consider cycle hypergraphs. 
The goal of this section is to prove the following result. 
\begin{Theorem}
  \label{cycles}
  Let $\H$ be a cycle hypergraph. 
  Then $\ara (\H) = \triara (\H) = \pd (\H)$. 
%  Then $\ara (\H) = \pd (\H)$. 
%
%  \par
%  Moreover, $\ara (\H)$ elements which generate $I(\H)$ up to radical 
%  can be constructed in polynomials containing at most $3$ terms. 
\end{Theorem}

\par
We first consider the case where $\H$ contains at most $1$ closed vertex. 
\begin{Lemma}
  \label{cycle1closed}
  Let $\H$ be a cycle hypergraph. 
  If $\H$ contains at most $1$ closed vertex, 
  then $\triara (\H) = \biara(\H) = \pd(\H)$.
%  then $\ara(\H)=\pd(\H)$.
\end{Lemma}
If $\H$ does not contain any closed vertex, 
then $I(\H)$ is also the edge ideal of a cycle. 
In \cite[Propositions 2.2, 2.3 and 2.4]{BKMY}, Barile et al.\  constructed 
binomials and monomials which generate this ideal up to radical. 
Below we show that  the same construction with minor modifications works 
also for $\H$ which contains precisely one closed vertex.
\begin{proof}[Proof of Lemma \ref{cycle1closed}]
  Let $\H$ be a $\mu$-cycle. 
  By assumption, we may assume that the monomial generators of $I(\H)$ 
  are following forms: 
  \begin{displaymath}
    y x_1 x_\mu, x_1 x_2, x_2 x_3, \ldots, x_{\mu - 1} x_{\mu}, 
  \end{displaymath}
  where $x_1, x_2, \ldots, x_{\mu}$ are pairwise distinct variables 
  and $y$ is either a variable which is different from 
  $x_1, x_2, \ldots, x_{\mu}$ or $y=1$.   By Theorem \ref{LMpd}, we have 
  \begin{displaymath}
    \pd (\H) 
    = \mu - \left( 1 + \left\lfloor \frac{\mu - 2}{3} \right\rfloor \right). 
  \end{displaymath}
  We distinguish three cases. 

  \par
  \textit{Case 1}: $\mu = 3m$ ($m \geq 1$). 

  \par
  In this case, $\pd (\H) = 2m$. 
  Consider the following $2m$ elements: 
  \begin{displaymath}
    \begin{aligned}
      &\left\{ 
      \begin{aligned}
        q_{0} &= x_{1} x_{2}, \\
        q_{1} &= y x_{1} x_{\mu} + x_{2} x_{3}, 
      \end{aligned}
      \right. \\
      &\left\{ 
      \begin{aligned}
        q_{2i} &= x_{3i+1} x_{3i+2}, \\
        q_{2i+1} &= x_{3i} x_{3i+1} + x_{3i+2} x_{3i+3}, 
      \end{aligned}
      \right. 
      \qquad i=1, 2, \ldots, m-1. 
    \end{aligned}
  \end{displaymath}
  Lemma \ref{SVLemma} 
    (see also \cite[Proposition 2.2]{BKMY}) yields that $q_0, q_1, \ldots, q_{2m-1}$ 
  generate $I(\H)$ up to radical.

  \par
  \textit{Case 2}: $\mu = 3m+1$ ($m \geq 1$). 

  \par
  In this case, $\pd (\H) = 2m+1$. 
  Consider the following $2m$ elements: 
  \begin{displaymath}
    \left\{ 
    \begin{aligned}
      q_{2i} &= x_{3i+2} x_{3i+3}, \\
      q_{2i+1} &= x_{3i+1} x_{3i+2} + x_{3i+3} x_{3i+4}, 
    \end{aligned}
    \right. 
    \qquad i=0, 1, 2, \ldots, m-1. 
  \end{displaymath}
  Set $q_{2m} = y x_1 x_{3m+1}$. 

  \par
Lemma \ref{SVLemma} 
  (see also \cite[Proposition 2.3]{BKMY}) now yields that $q_0, q_1, \ldots, q_{2m}$ 
  generate $I(\H)$ up to radical.

  \par
  \textit{Case 3}: $\mu = 3m+2$ ($m \geq 1$). 

  \par
  In this case, $\pd (\H) = 2m+1$. 
  Consider the following $2m$ elements: 
  \begin{displaymath}
    \begin{aligned}
      &\left\{ 
      \begin{aligned}
        q_{0} &= x_{1} x_{2}, \\
        q_{1} &= x_{2} x_{3} + x_{4} x_{5}, 
      \end{aligned}
      \right. \\
      &\left\{ 
      \begin{aligned}
        q_{2i} &= x_{3i} x_{3i+1} + x_{3i+2} x_{3i+3}, \\
        q_{2i+1} &= x_{3i+2} x_{3i+3} + x_{3i+4} x_{3i+5}, 
      \end{aligned}
      \right. 
      \qquad i=1, 2, \ldots, m-1. 
    \end{aligned}
  \end{displaymath}
  Set $q_{2m} = y x_1 x_{3m+2} + x_{3m} x_{3m+1}$ 
  (see also \cite[Proposition 2.4]{BKMY}).

  \par
  Set $J = (q_0, q_1, \ldots, q_{2m})$. 
  We claim $\sqrt{J} = I(\H)$. 
  It is clear that $J \subset I(\H)$. Thus we prove $\sqrt{J} \supset I(\H)$. 

  \par
  We first prove $x_1 I(\H) \subset \sqrt{J}$. 
  Since one has $q_0, q_1 \in J$, then
  $x_1 \cdot x_1 x_2, x_1 x_2 x_3, x_1 x_4 x_5 \in \sqrt{J}$. 
  We claim that 
  \begin{equation}
    \label{edgecycle1}
    x_1 x_{3i} x_{3i+1}, x_1 x_{3i+2} x_{3i+3}, x_1 x_{3i+4} x_{3i+5} 
    \in \sqrt{J}, \qquad i=1, 2, \ldots, m-1. 
  \end{equation}
  We prove this by induction on $i$. 

  \par
  For the case $i=1$, we need to prove that
  $x_1 x_3 x_4, x_1 x_5 x_6, x_1 x_7 x_8 \in \sqrt{J}$. 
  Since $x_1 q_2 = x_1 x_3 x_4 + x_1 x_5 x_6 \in J$ 
  and $x_1 x_4 x_5 \in \sqrt{J}$, Lemma \ref{SVLemma} 
yields $x_1 x_3 x_4, x_1 x_5 x_6 \in \sqrt{J}$. 
  Then, since $x_1 q_3 = x_1 x_5 x_6 + x_1 x_7 x_8 \in J$ 
  and $x_1 x_5 x_6 \in \sqrt{J}$, 
we also have $x_1 x_7 x_8 \in \sqrt{J}$. 

  \par
  Assume that (\ref{edgecycle1}) is true for $i-1$. Then 
  since $x_1 q_{2i} = x_1 x_{3i} x_{3i+1} + x_1 x_{3i+2} x_{3i+3} \in J$ 
  and $x_1 x_{3i+1} x_{3i+2} = x_1 x_{3(i-1)+4} x_{3(i-1)+5} \in \sqrt{J}$, Lemma \ref{SVLemma} yields
 $x_1 x_{3i} x_{3i+1}, x_1 x_{3i+2} x_{3i+3} \in \sqrt{J}$ 
  . 
  Then $x_1 q_{2i+1} = x_1 x_{3i+2} x_{3i+3} + x_1 x_{3i+4} x_{3i+5} \in J$ 
  and $x_1 x_{3i+2} x_{3i+3} \in \sqrt{J}$, hence we have 
 $x_1 x_{3i+4} x_{3i+5} \in \sqrt{J}$, as required. 

  \par
  Therefore (\ref{edgecycle1}) holds true for all $i$. 
  Moreover, $q_{2m} = y x_1 x_{3m+2} + x_3 x_{3m+1} \in J$ and 
  $x_1 x_{3m+1} x_{3m+2} = x_1 x_{3(m-1)+4} x_{3(m-1)+5} \in \sqrt{J}$. 
  These two facts imply
  $x_1 \cdot y x_1 x_{3m+2}, x_1 x_{3m} x_{3m+1} \in \sqrt{J}$. 

  \par
  Hence we have $x_1 I(\H) \subset \sqrt{J}$. 

  \par
  Next we prove $I(\H) \subset \sqrt{J}$. 
Since $x_1 I(\H) \subset \sqrt{J}$, we have $y x_1^2 x_{3m+2} \in \sqrt{J}$, whence $yx_1x_{3m+2}\in \sqrt{J}$. Since
$q_{2m} \in J$, we also have $x_{3m} x_{3m+1} \in \sqrt{J}$. 
  We now prove 
  \begin{equation}
    \label{edgecycle2}
    x_{3i} x_{3i+1},\; x_{3i+2} x_{3i+3},\; x_{3i+4} x_{3i+5} 
    \in \sqrt{J}, \qquad i=1, 2, \ldots, m-1 
  \end{equation}
  by descending induction on $i$. 

  \par
  When $i=m-1$, since $x_{3m} x_{3m+1} \in \sqrt{J}$ and 
  $q_{2(m-1)+1} = x_{3m-1} x_{3m} + x_{3m+1} x_{3m+2} \in J$, 
Lemma \ref{SVLemma} gives
  $x_{3m-1} x_{3m}, x_{3m+1} x_{3m+2} \in \sqrt{J}$. 
  Also, since $q_{2(m-1)} = x_{3m-3} x_{3m-2} + x_{3m-1}x_{3m} \in J$, 
  we have $x_{3m-3}x_{3m-2} \in \sqrt{J}$. 

  \par
  Next, assume that (\ref{edgecycle2}) holds true for $i+1$. Since $q_{2i+1} = x_{3i+2} x_{3i+3} + x_{3i+4} x_{3i+5} \in J$ 
  and $x_{3i+3} x_{3i+4} = x_{3(i+1)} x_{3(i+1)+1} \in \sqrt{J}$, 
  then Lemma \ref{SVLemma} yields $x_{3i+2} x_{3i+3}, x_{3i+4} x_{3i+5} \in \sqrt{J}$.
  Then $q_{2i} = x_{3i} x_{3i+1} + x_{3i+2} x_{3i+3} \in J$, 
 and so we have $x_{3i} x_{3i+1} \in \sqrt{J}$, as required. 

  \par
  Note that $x_1 x_2 = q_0 \in J$. 
  Also, since $q_1 = x_2 x_3 + x_4 x_5$ and $x_3 x_4 \in \sqrt{J}$, then 
  $x_4 x_5 \in \sqrt{J}$. This completes the proof. 
\end{proof}

\par
Next, we consider the case where the number of vertices is at most $4$. 
In this case, we know $\ara (\H) = \pd (\H)$ by \cite{KTY}. We prove the following slightly more precise lemma.
\begin{Lemma}
  \label{4cycle}
  Let $\H$ be a cycle hypergraph on $[\mu]$ with $\mu \leq 4$, 
  then $\triara (\H) = \biara (\H) = \pd (\H)$.
%  If $\H$ is a cycle with $\mu \leq 4$ vertices, then $\ara (\H)=\pd (\H)$.
%  Moreover, $\ara (\H)$ elements which generate $I(\H)$ up to radical 
%  can be constructed in polynomials containing at most $3$ terms. 
\end{Lemma}
\begin{proof}

  We first assume that $\pd (\H) = \mu$. In this case, 
  we can choose $\mu$ monomial generators. 
  Next we assume that $\pd (\H) < \mu$. 
  In this case, we can easily check that $\pd (\H) = \mu - 1$. 
  
  \par
  When $\mu = 3$, then the $3$ generators of $I(\H)$ can be written as 
  $x_1 x_2, y_1 x_1 x_3, y_2 x_2 x_3$, where each $y_i$ can possibly be $1$.
  By Lemma \ref{SVLemma},   
  $x_1 x_2, y_1 x_1 x_3 + y_2 x_2 x_3$ generate $I(\H)$ up to radical.

  \par
  When $\mu = 4$, then the $4$ generators of $I(\H)$ can be written as 
  $x_1 x_2, y_1 x_1 x_4, y_2 x_2 x_3, y_3 x_3 x_4$, 
  where each $y_i$ is possibly $1$. 
  Lemma \ref{SVLemma} yields that the elements
  $x_1 x_2, y_1 x_1 x_4 + y_2 x_2 x_3, y_3 x_3 x_4$ generate 
  $I(\H)$ up to radical. 
%
%  \par
%  When $\pd (\H) < \mu$, then we can easily check that $\pd (\H) = \mu - 1$ 
\end{proof}

\par
Thus, we can assume that the number of vertices of a cycle hypergraph 
is at least $5$. 
\begin{Lemma}
  \label{cycle2ad-closed}
  Let $\H$ be a cycle hypergraph on $[\mu]$ with $\mu \geq 5$. 
  If $\H$ contains two adjacent closed vertices, then 
  $\triara (\H) = \biara (\H) = \pd (\H)$.
%  Let $\H$ be a cycle hypergraph with $\mu \geq 5$ vertices. 
%  If $\H$ contains adjacent two closed vertices, then $\ara(\H)=\pd(\H)$.
\end{Lemma}
%{\bf Idea of the proof.} Essentially it follows by \cite[Lemma~4.7]{LM} using arguments similar to the ones of Section 1.
%\QED
%\bigskip
\begin{proof}
Without loss of generality we may assume $1$ and $\mu$ are two adjacent closed vertices.

  \par
  We first assume that the vertex $2$ is also closed. 
  Then we have 
  $\pd (\H) = \pd (\mathcal{H}_1) + 1$, by Lemma \ref{LMpd-inductive}. 
  Since ${\mathcal{H}}_1$ is a string hypergraph, we have 
  $\biara (\mathcal{H}_1) = \pd (\mathcal{H}_1)$, by Theorem \ref{strings}. 
  Now, the equality $\biara (\H) = \pd (\H)$ follows because the monomial $m_1$ corresponding to the vertex $1$, 
  together with  elements 
  which generate $I (\mathcal{H}_1)$ up to radical, 
  generate $I (\H)$ up to radical (i.e. if $\sqrt{I (\mathcal{H}_1)}=\sqrt{(a_1,\ldots,a_r)}$, then $\sqrt{I (\H)}=\sqrt{(m_1,a_1,\ldots,a_r)}$). 

  \par
  We may then assume that the vertex $2$ is open. 
  Then the  monomials corresponding to the vertices $1, 2, 3$ can be written as 
  $y_1 x_1 x_{\mu}, x_1 x_2, y_3 x_2 x_3$, respectively, 
  where $y_3$ is possibly $1$. 
  Note that $\mathcal{Q}_1$ is the disjoint union of $\H_3$ 
  and a closed vertex. Thus, 
  $\pd (\mathcal{Q}_1) = \pd (\mathcal{H}_3) + 1$. 
  By Lemma \ref{LMpd-inductive}, we have 
  \begin{displaymath}
    \pd (\H) = \max \{ \pd (\mathcal{H}_1), \pd (\mathcal{Q}_1) + 1 \} 
             = \max \{ \pd (\mathcal{H}_1), \pd (\mathcal{H}_3) + 2 \}. 
  \end{displaymath}
  Since $\H_1$ is a string hypergraph, 
  we have $\pd (\H_1) \leq \pd (\H_3) + 2$ 
  by Lemma \ref{LMpd-string-inductive-2}, 
  and thus $\pd (\H) = \pd (\mathcal{H}_3) + 2$. 
  Also, since $\mathcal{H}_3$ is a string hypergraph, Theorem \ref{strings} shows that 
 $\biara (\mathcal{H}_3) = \pd (\mathcal{H}_3)$. 
  
 Since the elements $x_1 x_2, y_1 x_1 x_{\mu} + y_3 x_2 x_3$, together with 
  elements which generate $I_3$ up to radical, 
  generate $I (\H)$ up to radical, 
  we obtain $\biara (\H) = \pd (\H)$. 
\end{proof}

In order to prove the following lemma, we use Theorem \ref{LMpd}.  
\begin{Lemma}
  \label{cycle3ad-open-0mod3}
  Let $\H$ be a cycle hypergraph. 
  Suppose that there is a string of opens with $n_0$ open vertices,
  with $n_0 \equiv 0 \mod 3$ in $\H$. 
  Then $\triara (\H) = \biara (\H) = \pd (\H)$.
%  Then $\ara (\H) = \pd (\H)$.
\end{Lemma}
\begin{proof}
  By Lemma \ref{cycle1closed}, 
  we may assume that $\H$ contains at least $2$ closed vertices. 
  Let $\s_0$ be the string of opens with $n_0$ open vertices,
  and let $u_1, u_2, u_3$ be three adjacent open vertices in $\s_0$ 
%  in the string of opens with length $n_0$ 
  such that $u_1$ is adjacent to a 
  closed vertex $v$. 
  Let $v'$ be the other neighbor of $u_3$. 
  We consider the ideal $I''$ with 
  $G(I'') = G(I) \setminus \{ u_1, u_2, u_3 \}$. 
  Then, $\H'' := \H (I'')$ is a string hypergraph 
  whose endpoints are $v$ and $v'$ 
  (i.e., $\H''$ is obtained by {\it deletion} 
   of the vertices $u_1,u_2$ and $u_3$ from $\H$ 
   and changing $v'$ to be closed if $v'$ is open in $\H$). 
  We claim that $\pd (\H) = \pd (\H'') + 2$. 
  Then, since we know that $\biara(\H'')=\ara (\H'') = \pd (\H'')$, 
  we can conclude that $\biara(\H)=\ara (\H) = \pd (\H)$, 
  because $\ara (\H'')$ elements which generate $I''$ up to radical, 
  together with $u_2$ and $u_1 + u_3$, generate $I$ up to radical. 

  \par
Hence, we only need to prove the equality $\pd (\H) = \pd (\H'') + 2$. 
  We first note that $\mu (\H'') = \mu (\H) - 3$ and that $v'$ is a closed vertex in $\H''$ (independently of 
  whether it is closed or not in $\H$). 
  
  If $v'$ is closed in $\H$, then $s(\H'') = s(\H) - 1$. 
  Since $\lfloor (n_0 - 1)/3 \rfloor = 0$, we have $b(\H'') = b(\H) - 1$. 
  Moreover, $M(\H'') = M(\H)$, because $\s_0$ 
%  the string of opens with length $n_0$ 
  does not belong to any $2$-special configuration in $\H$. 
  Therefore, we have $\pd (\H) = \pd (\H'') + 2$, by Theorem \ref{LMpd}. 
  
  If $v'$ is open in $\H$, then $s(\H'') = s(\H)$. 
  Let $n_0''$ be the number of open vertices in the string of opens $\H''$, one of whose endpoints is $v'$. Then, $n_0'' = n_0 -  4  \equiv 2 \mod 3$. 
  Note that $\lfloor (n_0 - 1)/3 \rfloor = n_0 / 3 -1$ 
  and $\lfloor (n_0'' - 1)/3 \rfloor = n_0/3 - 2$. 
  Thus, $b(\H'') = b(\H) - 1$. 
  Moreover, we have $M (\H'') = M (\H)$, because both strings of opens do 
  not belong to any $2$-special configuration. 
  Therefore,  by Theorem \ref{LMpd}, we have $\pd (\H) = \pd (\H'') + 2$. 
\end{proof}

By Lemma \ref{cycle3ad-open-0mod3}, we may then assume that 
each string of opens in $\H$ contains a number of open vertices that is either congruent to $2 \mod 3$ or $1 \mod 3$. 
\begin{Lemma}
  \label{cycle3ad-open}
  If we prove that $\ara (\H) = \triara (\H) = \pd (\H)$ 
  for a cycle hypergraph $\H$ 
  whose strings of opens all have at most $2$ open vertices, 
  then Theorem \ref{cycles} follows. 
\end{Lemma}
\begin{proof}
  Let $\H$ be a $\mu$-cycle. 
  By Lemma \ref{cycle1closed}, we may assume $\H$ has at least 
  two closed vertices. 
  By Lemma \ref{4cycle}, we may assume $\mu \geq 5$. 
  Moreover, by Lemma  \ref{cycle2ad-closed}, 
  we may assume that there are no two adjacent closed vertices in $\H$. 

  \par
  Suppose that $\H$ contains a string of opens $\s$ 
  with $n_0 \geq 3$ open vertices. 
By Lemma \ref{cycle3ad-open-0mod3}, we may assume that $n_0 \equiv 1,2 \mod 3$. 

  \par
  We first assume that $n_0 \equiv 1 \mod 3$. 
  Let $v$ be an endpoint of $\s$, and let $u_1, u_2, u_3, u_4$ be adjacent open vertices 
  following $v$. Let $\H'$ be the cycle hypergraph obtained by turning 
  $u_2$ into a closed vertex. 
  We claim that $\pd (\H) = \pd (\H')$. 

  \par
Indeed, by the change we made, the string of opens $\s$ in $\H$ is now divided into two 
  strings of opens $\s_1$ and $\s_2$ (in $\H'$), with $1$ and $n_0 -2$ open vertices, respectively.
  It is easy to see that $\mu (\H') = \mu (\H)$, $s (\H') = s (\H) + 1$. 
  Also, since $\lfloor (n_0 - 1)/3 \rfloor = (n_0 - 1)/3$, 
  $\lfloor (1 - 1)/3 \rfloor + \lfloor ((n_0 - 2) - 1)/3 \rfloor 
    = (n_0 - 1)/3 - 1$, 
  we have $b (\H') = b (\H)$. 
  Moreover, the modularity is also unchanged because the change does not 
  affect to the number of $2$-special configurations. Now, the equality $\pd (\H) = \pd (\H')$ follows from the formula of Theorem \ref{LMpd}.

  \par
  Next, assume that $n_0 \equiv 2 \mod 3$. 
  Let $v$ be an endpoint of $\s$ and let $u_1, u_2, u_3, u_4, u_5$ 
   be adjacent open vertices following $v$. 
  Let $\H'$ be the cycle hypergraph obtained by turning 
  $u_3$ into a closed vertex. 
  We claim that $\pd (\H) = \pd (\H')$. 

  \par
  By the change, the string of opens $\s$ in $\H$ is now divided into two 
    strings of opens $\s_1$ and $\s_2$ (in $\H'$), with $2$ and $n_0 -3$ open vertices, respectively. 
  It is easy to see that $\mu (\H') = \mu (\H)$, $s (\H') = s (\H) + 1$. 
  Since $\lfloor (n_0 - 1)/3 \rfloor = (n_0 - 2)/3$, 
  $\lfloor (2 - 1)/3 \rfloor + \lfloor ((n_0 - 3) - 1)/3 \rfloor 
    = (n_0 - 2)/3 - 1$, 
  we have $b (\H') = b (\H)$. 
  Furthermore, the modularity is also unchanged because the change does not 
  affect to the number of $2$-special configurations.  All the above together with the formula of Theorem \ref{LMpd} implies the equality $\pd (\H) = \pd (\H')$.
  
  \par
Moreover, in either case, if $\triara (\H') = \pd (\H')$ then also $\triara (\H) = \pd (\H)$ holds, as it can be seen by substituting $1$ for the variables corresponding to the vertices which we made become closed.

Then, this procedure produces a new hypergraph $\widetilde{\H}$ 
(obtained by making selected open vertices of $\H$ become closed) 
and all strings of opens in $\widetilde{\H}$ have at most $2$ two open 
vertices. 
Moreover, the above shows that if 
$\triara (\widetilde{\H})=\pd (\widetilde{\H})$, 
then $\triara (\H) = \pd (\H)$ also holds. 
%Moreover, the above shows that $\pd (\H) = \pd (\widetilde{\H})$, 
%and  $\triara (\widetilde{\H}) = \triara (\H)$. 
The statement now follows.
\end{proof}

By the above results, we may then assume that $\H$ is a cycle 
not containing two consecutive closed vertices and 
whose strings of opens  have at most $2$ open vertices.
Note that for such a graph, $b(\H) = s(\H)$ holds. 

Next, we prove the case where there are strings of opens with precisely $2$ open vertices.
\begin{Lemma}
  \label{cycle-cooc}
  Assume that $\H$ contains a closed--open--open--closed string $\s$, 
  where the two closed vertices of $\s$ are distinct. 
  Let $\H'$ be the cycle hypergraph obtained by removing the
  $2$ open vertices of $\s$ from $\H$
  and identifying the two closed vertices of $\s$.

  \par
  If $\triara (\H') = \pd (\H')$, then $\triara (\H) = \pd (\H)$.
\end{Lemma}
\begin{proof}
  Let $\H$ be a cycle hypergraph on $[\mu]$. By Lemma  \ref{cycle2ad-closed}, 
    we may assume that there are no two adjacent closed vertices in $\H$, 
  and by Lemma \ref{cycle3ad-open} all strings of opens have at most 
  two open vertices. 
  We first claim that $\pd (\H) = \pd (\H') + 2$. 

  \par
  It is easy to see that $\mu (\H) = \mu (\H') + 3$ and 
  $s (\H) = s (\H') + 1$. 
  Since the removed string of opens has $2$ open vertices, 
  the modularity is unchanged. Hence, 
  the claim follows by the formula of Theorem \ref{LMpd}.
To prove the statement we show that $\triara(\H) \leq \triara(\H')+2$.

  \par
  Let $1, \mu, \mu - 1, \mu - 2$ be the vertices of the string $\s$. 
  We set the monomials corresponding to these vertices 
  to be 
  \begin{equation}
    \label{4monomials}
    y_1 x_1 x_{\mu}, x_{\mu - 1} x_{\mu}, 
    x_{\mu - 2} x_{\mu - 1}, y_{\mu - 2} x_{\mu - 3} x_{\mu - 2}. 
  \end{equation}
  We set
  \begin{displaymath}
    \left\{
    \begin{aligned}
      g_0 &= y_1 y_{\mu - 2} x_1 x_{\mu - 3} x_{\mu - 2} x_{\mu}, \\
      g_1 &= y_1 x_1 x_\mu + x_{\mu - 2} x_{\mu - 1}, \\
      g_2 &= y_{\mu - 2} x_{\mu - 3} x_{\mu - 2} + x_{\mu - 1} x_{\mu}. 
    \end{aligned}
    \right.  
  \end{displaymath}
We claim that
  \begin{displaymath}
    y_1 x_1 x_{\mu},\; x_{\mu - 1} x_{\mu}, \;
    x_{\mu - 2} x_{\mu - 1},\; y_{\mu - 2} x_{\mu - 3} x_{\mu - 2}\; 
    \in \sqrt{(g_0, g_1, g_2)}. 
  \end{displaymath}
Indeed, since   $$x_{\mu - 2} x_{\mu - 1} \cdot x_{\mu - 1} x_{\mu} 
     = (g_1 - y_1 x_1 x_{\mu})(g_2 - y_{\mu - 2} x_{\mu - 3} x_{\mu - 2}) 
    \in (g_0, g_1, g_2),$$
  we have $x_{\mu - 2} x_{\mu - 1} x_{\mu} \in \sqrt{(g_0, g_1, g_2)}$. 
  Then the claim follows by Lemma \ref{SVLemma}. 

  \par
%  Let $J$ be an ideal with $\sqrt{J} = I(\H')$. 
%  Then $y_1 y_2 x_1 x_\mu x_3 x_4 \in \sqrt{J}$. 
%  Hence $\sqrt{J + (g_1, g_2)} = I(\H)$. 
  Let $I_0$ be the squarefree monomial ideal which is generated by 
  all monomials in $G(I(\H))$ except for the $4$ monomials in (\ref{4monomials}). 
  Then $I(\H) = I_0 + (y_1 x_1 x_{\mu}, x_{\mu - 1} x_{\mu}, 
    x_{\mu - 2} x_{\mu - 1}, y_{\mu - 2} x_{\mu - 3} x_{\mu - 2})$. 
  Let $I'$ be the squarefree monomial ideal defined as 
  $I'=I_0+(y_1y_{\mu - 2} x_1 x_{\mu - 3} x_{\mu - 2}x_{\mu})$ 
  and note that $\H (I') = \H'$. 
  Since $g_0=y_1 y_{\mu - 2} x_1 x_{\mu - 3} x_{\mu - 2} x_{\mu} \in I'$ 
  it follows that $\ara (\H')$ elements which generate $I'$ up to radical, 
  together with $g_1, g_2$ generate $I (\H)$ up to radical. 
\end{proof}

Therefore, we reduce to the case of cycle hypergraphs 
in which closed vertices and open vertices appear alternately. 
\begin{Lemma}
  \label{cycle-co}
  Let $\H$ be a cycle hypergraph in which closed vertices and open vertices 
  appear alternately. 
  Then we have $\ara (\H) = \triara (\H) = \pd (\H)$. 
\end{Lemma}
\begin{proof}

  We first note that the number $\mu$ of vertices of $\H$ is even. 

  \par \noindent
  \textit{Case 1}: 
  $\mu = 4m$. 

  \par
  By Theorem \ref{LMpd} we have $\pd (\H) = 3m$, because $s (\H) = 2m$ and $M (\H) = m$. 
  We now divide the vertices in disjoint groups of $4$ adjacent vertices. 
  In other words, there exist $m$ strings of the shape 
  {\it closed--open--closed--open} in $\H$. 
  It suffices to show that the ideal associated to any such string 
  is generated up to radical by $3$ polynomials 
  each of which has at most $3$ terms. 
  So, let $m_1, m_2, m_3, m_4$ be monomials corresponding to the 
  $4$ vertices of the string, we can write $m_1, m_2, m_3, m_4$ as 
  $y_1 x_1 x_{\mu}, x_1 x_2, y_3 x_2 x_3, x_3 x_4$. 
  By Lemma \ref{SVLemma},  the following $3$ polynomials 
  \begin{displaymath}
    x_1 x_2, \  y_1 x_1 x_{\mu} + y_3 x_2 x_3, \  x_3 x_4 
  \end{displaymath}
  generate $(m_1, m_2, m_3, m_4)$ up to radical, whence the statement follows.  
  \par
  \bigskip

  \par \noindent
  \textit{Case 2}: $\mu = 4m + 2$ ($m \geq 1$). 

  \par
  In this case, we prove the statement by induction on $m$. 
  First assume $m=1$. 
  Then $I(\H)$ is generated by the following $6$ monomials: 
  \begin{displaymath}
    y_1 x_1 x_6, \  x_1 x_2, \  y_3 x_2 x_3, \  
    x_3 x_4, \  y_5 x_4 x_5, \  x_5 x_6. 
  \end{displaymath}
  By Theorem \ref{LMpd}, we have $\pd (\H) = 4$, 
  and by Lemma \ref{SVLemma} the following $4$ polynomials 
  generate $I(\H)$ up to radical: 
  \begin{displaymath}
    \left\{
    \begin{aligned}
      &x_1 x_2, \\
      &x_3 x_4, \\
      &x_5 x_6, \\
      &y_1 x_1 x_6 + y_3 x_2 x_3 + y_5 x_4 x_5. 
    \end{aligned}
    \right. 
  \end{displaymath}

  \par
  Now we assume that $m \geq 2$. 
  In this case, $\H$ contains a $2$-special configuration $\s$: 
  {\it closed--open--closed--open--closed}. 
  Let $1, \mu, \mu - 1, \mu - 2, \mu - 3$ be the vertices of the string $\s$. 
  We set the monomials corresponding to these vertices to be
  \begin{equation}
    \label{5monomials}
    y_1 x_1 x_{\mu},\  x_{\mu - 1} x_{\mu},\  
     y_{\mu - 1} x_{\mu - 2} x_{\mu - 1},\  x_{\mu - 3} x_{\mu - 2},\  
     y_{\mu - 3} x_{\mu - 4} x_{\mu - 3}.
  \end{equation}
  Let $\H'$ be the cycle hypergraph obtained by removing the $3$ inner vertices 
  $\mu, \mu - 1, \mu - 2$ of $\s$ from $\H$ 
  and identifying the two endpoints $1$ and $\mu - 3$ of $\s$. 
  Then $\H'$ is the cycle hypergraph with $\mu (\H') = 4(m-1)+2$
  in which closed vertices and open vertices appear alternately. 
  Note that Theorem \ref{LMpd} yields $\pd (\H) = \pd (\H') + 3$, because 
  $\mu (\H) = \mu (\H') +4$, $s (\H) = s (\H') + 2$, 
  and $M (\H) = M(\H') + 1$. 

  \par
  Let $I_0$ be the squarefree monomial ideal which is generated by 
  all monomials in $G(I(\H))$ except for the $5$ monomials 
  in (\ref{5monomials}). Then 
  \begin{displaymath}
    I(\H) = I_0 + (y_1 x_1 x_{\mu},\ x_{\mu - 1} x_{\mu},\ 
    y_{\mu - 1} x_{\mu - 2} x_{\mu - 1},\  x_{\mu - 3} x_{\mu - 2},\ 
    y_{\mu - 3} x_{\mu - 4} x_{\mu - 3}). 
  \end{displaymath}
  We set $I' = I_0 + (y_1 y_{\mu - 3} x_1 x_{\mu - 4} x_{\mu - 3} x_{\mu})$. 
  Note that $\H (I') = \H'$ and
  $y_1 y_{\mu - 3} x_1 x_{\mu - 4} x_{\mu - 3} x_{\mu}$ 
  is the monomial corresponding to the vertex $1$ of $\H (I')$. 
  Since $y_1 y_{\mu - 3} x_1 x_{\mu - 4} x_{\mu - 3} x_{\mu} \in I'$, 
  the following $3$ polynomials, together with $\ara (\H')$ elements 
  which generate $I'$ up to radical, generate $I (\H)$ up to radical: 
  \begin{displaymath}
    \left\{ 
    \begin{aligned}
      &x_{\mu - 1} x_{\mu}, \\
      &x_{\mu - 3} x_{\mu - 2}, \\
      &y_1 x_1 x_{\mu} + y_{\mu - 1} x_{\mu - 2} x_{\mu - 1} 
       + y_{\mu - 3} x_{\mu - 4} x_{\mu - 3}. 
    \end{aligned}
     \right. 
  \end{displaymath}
\end{proof}

%%%%%%%%%%%%%%%%%%%%%%%%%%%%%%%%%%%%%%%
\begin{acknowledgement}
  The authors thank the referee for reading our manuscript carefully. 
\end{acknowledgement}

%%%%%%%%%%%%%%%%%%%%%%%%%%%%%%%%%%%%%%%%%%%%%%%%%

\end{document}